\documentclass[12pt,a4paper]{amsart}
\usepackage[backref=page,colorlinks]{hyperref} 
\usepackage[centertags]{amsmath}
\usepackage{amsfonts}
\usepackage{graphics}
\usepackage{subcaption}
\usepackage{graphicx}
\usepackage{amssymb}
\usepackage{amsthm}
\usepackage{newlfont}
\usepackage{epsfig}
\usepackage{xcolor}
\usepackage{marginnote}
\usepackage[final]{pdfpages}
\usepackage{tikz}
\newcommand{\Lo}{\mathfrak{L}}
\newcommand{\Gain}{\mathfrak{G}}
\newcommand{\CC}{\mathcal{C}}

\newtheoremstyle{theorem}
  {10pt}          
  {10pt}  
  {\sl}  
  {\parindent}     
  {\bf}  
  {. }    
  { }    
  {}     
\theoremstyle{theorem}
\newtheorem{theorem}{Theorem}
\newtheorem{corollary}[theorem]{Corollary}
\newtheorem{lem}[theorem]{Lemma}
\newtheorem{claim}[theorem]{Claim}

\newtheorem{rem}[theorem]{Remark}
\newtheorem{defn}{Definition}[section]
\newtheoremstyle{defi}
  {8pt}          
  {8pt}  
  {\rm}  
  {\parindent}     
  {\bf}  
  {. }    
  { }    
  {}     
\theoremstyle{defi}

\usepackage{latexsym}
\usepackage{color}
\usepackage{float}
\usepackage{amssymb}
\usepackage{amsmath}
\usepackage{amsthm} \usepackage{graphicx}
\newcommand{\Gr}{\mathcal{G}}
\newcommand{\G}{\gamma}
\newcommand{\Ud}{\mathbb{U}_D}
\newcommand{\Di}{\Delta}
\newcommand{\di}{\delta}
\newcommand{\A}{\alpha}
\newcommand{\LE}{\preceq}

\newtheorem{lemma}[theorem]{Lemma}

\newtheorem{conjecture}{Conjecture}

\date{}
\begin{document}
\title[Wiener index of unicyclic graphs]{On minimizing the Wiener index of unicyclic graphs with fixed girth and given degree sequence}
\maketitle
\begin{center}
\author{Alewyn P. Burger$^{1}$}

\email{apburger@sun.ac.za} \\ 	     
\author{Valisoa R. M. Rakotonarivo$^{2,3}$\\}
\email{valisoa.rakotonarivo@up.ac.za\\}
\address{$^1$ Department of Logistics, Stellenbosch University, South Africa
\\$^2$ Department of Mathematics and Applied Mathematics,\\ University of Pretoria, South Africa\\
$^3$ DSI-NRF, Centre of Excellence in Mathematical And Statistical Sciences, South Africa}
\date{\Today}
\end{center}

\begin{abstract}
The Wiener index of a graph is the sum of all the distances between any pair of vertices. We aim to describe graphs which minimize the Wiener index among all unicyclic graphs with fixed girth and given degree sequence. Depending on where the centroid of the graph is, we will present three candidates for the minimization, namely the greedy unicyclic graph, the cycle-centered graph and the out-greedy unicyclic graph.
\end{abstract}

%
\noindent
{\bf Keywords}: Wiener index, unicyclic graphs, degree sequence, fixed girth, centroid.

\section{Introduction}
\label{Intro}
The Wiener index, which is the sum of the distances between all pairs of vertices, was  conceived by the chemist Wiener in 1947 \cite{Wiener1947}, who introduced the idea to predict chemical and physical properties of a compound by analysing the structure of its molecules using techniques from graph theory. The Wiener index has attracted considerable attention from researchers in mathematics and chemistry, and it continues to be studied extensively. Many researchers have focused on the class of trees (see for example \cite{Dobrynin2001}), because many graphs arising from molecules are trees, and because they have a relatively simple structure. The Wiener index of trees with a given degree sequence has been studied by several authors, for example \cite{Wang2009,Zhang2008}. For unicyclic graphs (graphs obtained from trees by adding a single edge), the Wiener index was first considered by Gutman and al. \cite{Gutman1997}, later some authors investigated various restrictions such as given diameter \cite{Du2009}, given matching number \cite{Du2010,Chen2012,Cambie2018}, a given girth \cite{Yu2010}, a given number of pendant or cut vertices \cite{Tan2017}, and a given bipartition \cite{knor2015,Jiang2019}. To our knowledge, the Wiener index of unicyclic graphs with a given degree sequence has not been studied in the literature, and it is the main motivation of this paper to fill this gap.  Some of the recent works on unicyclic graphs for a distance-based invariant can be found in \cite{Feng2020}, where they focused on fixed diameters and in \cite{Zhang2024}, where they characterized unicyclic graph with given maximum degree. 

All graphs considered in this paper are simple, finite and undirected. Let $G$ be a graph, and $u,v$ be vertices of $G$. The \emph{distance} between $u$ and $v$, denoted by $d_{G}(u,v)$ (we may omit the $G$ if there is no confusion) is the length of the shortest path between $u$ and $v$. Moreover, the
\emph{distance of a vertex} $u$ in $G$, denoted by $D_G(u)$, is the sum of distances from $u$ and 
all other vertices of $G$. The subgraph of $G$ induced by the vertices with minimum
distance is called the \emph{centroid} of $G$, and is denoted by $M(G)$. The Wiener index can be formally written as follows:
\begin{align*}
W(G)=\sum_{\{u,v\}}d_G(u,v)=\frac{1}{2}\sum_{u \in V(G)}D_G(u).
\end{align*}


The paper is structured as follows:  in Section \ref{trees}, we will recall results obtained for trees with given degree sequence. Since trees and unicyclic graphs only differ by one edge, it is necessary to understand the structure of the tree that minimizes the Wiener index. Next, in Section \ref{cycle}, we will discuss the position of the vertices on the cycle itself according to their subtrees/degrees. Before attacking the vertices outside the cycle, we need to comprehend the influence of adding an edge to a tree to get a unicyclic graph, which will be in Section \ref{difference}. In Section \ref{outside}, we then figure out the position of the vertices not belonging to the cycle. Lastly, in section \ref{fixed girth}, we characterize the graphs that minimizes the Wiener index in the set $\Ud^{\G}$, which is the set of all unicyclic graphs with given degree sequence $D$ and fixed girth $\G$. 
\section{Known results on trees with given degree sequence}\label{trees}
Let $D=(d_1,d_2,\dots,d_t,1,\dots,1)$ be the degree sequence of a unicyclic graph, we assume without loss of generality that $d_1\geq d_2\geq \cdots d_t\geq 2$. To obtain a tree from a unicyclic graph or vice versa, we need to remove an edge (resp. add an edge). In terms of degree sequences, it means the sequences only differ within two entries in such a way:
\begin{align*}
d'_i&=d_i-1 \\
d'_j&=d_j-1, 
\end{align*}
where  $d_i$ and $d_j$ are the entries in $D$ and $d'_i$ and $d'_j$ are entries in the corresponding degree sequence $D'$ of a tree, the other degrees remain the same.
It is then natural to understand the characterization of trees that minimize the Wiener index among trees with given degree sequence, which is the aim of this section. 

Let $T$ be a tree rooted at one of its vertices $v$. The height of $T$, denoted by $h(T)$, is the largest distance between the root $v$ and any leaf and a \emph{pseudo-leaf} is a vertex whose neighbours are all leaves except possibly one.


\begin{defn}[\cite{Schmuck},Greedy tree]\label{greedytree}
Let $D$ be a degree sequence of a tree. The greedy tree denoted by $\Gr(D)$ is achieved through the following ``greedy algorithm'':
\begin{itemize}
 \item[ i)] Label the vertex with the largest degree as $v$ (the root);
 \item[ ii)] Label the neighbors of $v$ as $v_1$, $v_2$, \dots, assign the largest degrees available to them such that $\deg(v_1) \geq \deg(v_2) \geq \dots$;
 \item[ iii)] Label the neighbors of $v_1$ (except $v$) as $v_{11}, v_{12},$\dots such that they take all the largest degrees available and that $\deg(v_{11}) \geq \deg(v_{12}) \geq\dots$, then do the same for $v_2$, $v_3$, \dots;
 \item[ iv)] Repeat (iii) for all the newly labeled vertices, always start with the neighbors of the labeled vertex with largest degree whose neighbors are not labeled yet.
\end{itemize}
\end{defn}
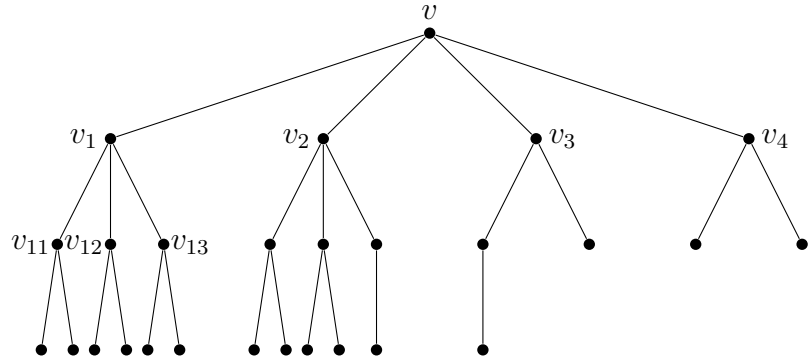
\begin{figure}[H]

\centering
    \begin{tikzpicture}[scale=0.7]
        \node[fill=black,circle,inner sep=1.5pt] (v) at (10,6) {};
        \node[fill=black,circle,inner sep=1.5pt] (v1) at (4,4) {};
        \node[fill=black,circle,inner sep=1.5pt] (v2) at (8,4) {};
        \node[fill=black,circle,inner sep=1.5pt] (v3) at (12,4) {};
        \node[fill=black,circle,inner sep=1.5pt] (v4) at (16,4) {};
        \node[fill=black,circle,inner sep=1.5pt] (v42) at (17,2) {};
        \node[fill=black,circle,inner sep=1.5pt] (v12) at (4,2) {};
        \node[fill=black,circle,inner sep=1.5pt] (v41) at (15,2) {};
        \node[fill=black,circle,inner sep=1.5pt] (v32) at (13,2) {};
        \node[fill=black,circle,inner sep=1.5pt] (v22) at (8,2) {};
        \node[fill=black,circle,inner sep=1.5pt] (v31) at (11,2) {};
        \node[fill=black,circle,inner sep=1.5pt] (v23) at (9,2) {};
        \node[fill=black,circle,inner sep=1.5pt] (v21) at (7,2) {};
        \node[fill=black,circle,inner sep=1.5pt] (v13) at (5,2) {};
        \node[fill=black,circle,inner sep=1.5pt] (v11) at (3,2) {};

        \node[fill=black,circle,inner sep=1.5pt] (v111) at (3.3,0) {};
        \node[fill=black,circle,inner sep=1.5pt] (v112) at (2.7,0) {};
        \node[fill=black,circle,inner sep=1.5pt] (v121) at (4.3,0) {};
        \node[fill=black,circle,inner sep=1.5pt] (v122) at (3.7,0) {};
        \node[fill=black,circle,inner sep=1.5pt] (v131) at (5.3,0) {};
        \node[fill=black,circle,inner sep=1.5pt] (v132) at (4.7,0) {};
        \node[fill=black,circle,inner sep=1.5pt] (v211) at (7.3,0) {};
        \node[fill=black,circle,inner sep=1.5pt] (v212) at (6.7,0) {};
        \node[fill=black,circle,inner sep=1.5pt] (v221) at (8.3,0) {};
        \node[fill=black,circle,inner sep=1.5pt] (v222) at (7.7,0) {};
        \node[fill=black,circle,inner sep=1.5pt] (v231) at (9,0) {};
        \node[fill=black,circle,inner sep=1.5pt] (v311) at (11,0) {};

        \draw (v)--(v1);
        \draw (v)--(v2);
        \draw (v)--(v3);
        \draw (v)--(v4);
        \draw (v1)--(v11);
        \draw (v1)--(v12);
        \draw (v1)--(v13);
        \draw (v2)--(v21);
        \draw (v2)--(v22);
        \draw (v2)--(v23);
        \draw (v3)--(v31);
        \draw (v3)--(v32);
        \draw (v4)--(v41);
        \draw (v4)--(v42);
        \draw (v11)--(v111);
        \draw (v11)--(v112);
        \draw (v12)--(v121);
        \draw (v12)--(v122);
        \draw (v13)--(v131);
        \draw (v13)--(v132);
        \draw (v21)--(v211);
        \draw (v21)--(v212);
        \draw (v22)--(v221);
        \draw (v22)--(v222);
        \draw (v23)--(v231);
        \draw (v31)--(v311);

\node at  (10,6.4) {$v$};

\node at  (3.5,4) {$v_{1}$};
\node at  (7.5,4) {$v_{2}$};
\node at  (12.5,4) {$v_{3}$};
\node at  (16.5,4) {$v_{4}$};

\node at  (2.5,2) {$v_{11}$};
\node at  (3.5,2) {$v_{12}$};
\node at  (5.5,2) {$v_{13}$};
    \end{tikzpicture}

\caption{A greedy tree (only the labels of the first height vertices are shown).}
\label{fig:greedy_tree}
\end{figure}

\begin{rem}\label{height}
We observe that in a greedy tree the leaves are all of distance $h(\Gr(D))$ or $h(\Gr(D))-1$ from the root.
\end{rem}

\begin{theorem}[\cite{Wang2009},\cite{Schmuck}]\label{ThmWang}
Given a degree sequence $D$ of a tree, the greedy tree $\Gr(D)$ is the unique tree that minimizes the Wiener index.
\end{theorem}

In \cite{Wang2009}, the author described extremal paths in an optimal tree (trees that minimizes the Wiener index) as follows. Let $u,w$ be any two leaves of $T$. The path $P_{uw}$ from $u$ to $w$ can be written as $u_k u_{k-1} \dots u_2 u_1 w_1 w_2 \dots w_{k-1}$ when the order of $|P_{uw}|$ is odd, or $u_ku_{k-1}\dots u_2u_1w_1w_2\dots w_{k-1}w_k$ when $|P_{uw}|$ is even, where $u_k = u,w_{k-1} = w$ (resp. $u_k=u,w_{k}=w$). If we remove all the edges in $P_{uw}$ from $T$ we obtain $2k-1$ (respectively $2k$)connected components. We denote by $U_i,W_i$ the components that contain respectively $u_i,w_i$ for $i=1,\dots,k$. Figure \ref{path Puv} shows such labelling with respect to a path of even order.

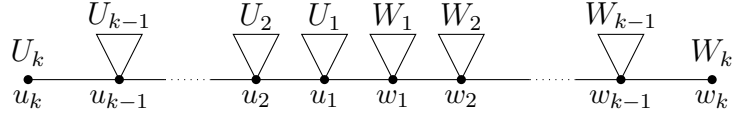
\begin{figure}[h!]
\centering
\begin{tikzpicture}[scale=1.2]
\draw [-] (1,-1) -- (2,-1) ;
\draw [-](2,-1)--(2.25,-0.5)--(1.75,-0.5)--(2,-1)--(2.5,-1);
\draw [dotted](2.5,-1)--(3,-1);
\draw [-](3,-1)--(3.5,-1)--(3.75,-0.5)--(3.25,-0.5)--(3.5,-1)--(4.25,-1);
\draw [-](4.25,-1)--(4.5,-0.5)--(4,-0.5)--(4.25,-1)--(5,-1);
\draw [-](4.25+0.75,-1)--(4.5+0.75,-0.5)--(4+0.75,-0.5)--(4.25+0.75,-1)--(5+0.75,-1);
\draw [-](4.25+1.5,-1)--(4.5+1.5,-0.5)--(4+1.5,-0.5)--(4.25+1.5,-1)--(5+1.5,-1);
\draw [dotted](6.5,-1)--(7,-1);
\draw [-](3+4,-1)--(3.5+4,-1)--(3.75+4,-0.5)--(3.25+4,-0.5)--(3.5+4,-1)--(4.5+4,-1);
\draw (1,-0.7) node{$U_{k}$};
\draw (2,-0.3) node{$U_{k-1}$};
\draw (7.5,-0.3) node{$W_{k-1}$};
\draw (8.5,-0.7) node{$W_{k}$};
\draw (3.5,-0.3) node{$U_2$};
\draw (4.25,-0.3) node{$U_{1}$};
\draw (5,-0.3) node{$W_{1}$};
\draw (5.75,-0.3) node{$W_{2}$};
\fill (3.5,-1) circle(0.05) node [below] {\small $u_{2}$};
\fill (4.25,-1) circle(0.05) node [below] {\small $u_{1}$};
\fill (5,-1) circle(0.05) node [below] {\small $w_{1}$};
\fill (5.75,-1) circle(0.05) node [below] {\small $w_{2}$};
\fill (1,-1) circle(0.05) node [below] {\small $u_{k}$};
\fill (2,-1) circle(0.05) node [below] {\small $u_{k-1}$};
\fill (8.5,-1) circle(0.05) node [below] {\small $w_{k}$};
\fill (7.5,-1) circle(0.05) node [below] {\small $w_{k-1}$};
\end{tikzpicture}		
\caption{Labelling of a path $P_{uw}$ of even order.}
\label{path Puv}
\end{figure}
 
The following lemma describes the \emph{concave} arrangement of the subtrees attached to each vertex along a maximal path according to its order.

\begin{lem}[\cite{Wang2009}, concavity]\label{optimalpath}
In an optimal tree, for a maximal path with labelling as Figure \ref{path Puv}, we have
\begin{equation*}
|V(U_1)| \geq |V(W_1)| \geq |V(U_2)| \geq |V(W_2)|\geq \cdots \geq |V(U_k)|=|V(W_k)|=1
\end{equation*}
if the path is of even order $2k$; and
\begin{align*}
&|V(U_1)| \geq |V(W_1)| \geq |V(U_2)| \geq |V(W_2)|\geq \cdots 
\\&\geq |V(U_{k-1})|\geq |V(W_{k-1})|=|V(U_{k})|=1
\end{align*}
if the path of odd order $2k-1$.

In particular:
\begin{equation*}
d(u_1) \geq d(w_1) \geq d(u_2) \geq d(w_2)\geq \cdots \geq d(u_k)=d(w_k)=1
\end{equation*}
if the path is of even order $2k$; and
\begin{equation*}
d(u_1) \geq d(w_1) \geq d(u_2) \geq d(w_2)\geq \cdots \geq d(u_{k-1})\geq d(w_{k-1})=d(u_{k})=1
\end{equation*}
if the path is of odd order $2k-1$.
\end{lem}

\begin{rem}\label{atmost2vces}
Note that Lemma \ref{optimalpath} still holds in the case of a maximal path in a unicyclic graph that involves at most 1 vertex of the cycle. In fact, for this case, the cycle itself can be considered as ``a tree component" and the switching argument in the proofs in \cite{Wang2009} still holds.
\end{rem}

Let $S_n$ be the set of all permutations of $\{1,\dots,n\}$. Let $A=(a_1,\dots,a_n)$ and $B=(b_1,\dots,b_n)$ be sequences of nonnegative numbers. We say that $B$ \textit{majorizes} $A$ if for all $1 \leq k \leq n$ we have
\[\sum_{i=1}^k a_i \leq \sum_{i=1}^k b_i.\]
If for any $\sigma \in S_n$ the sequence $B$  majorizes $(a_{\sigma(1)},\dots,a_{\sigma(n)})$, then we write
\[A \LE B.\]

In \cite{Zhang2008}, the authors compared trees of the same order with different degree sequences and proved that the greedy tree corresponding to the sequence which majorizes any other sequence has the minimum Wiener index. 
\begin{theorem}[\cite{Zhang2008}]
Let $B$ and $B'$ be the degree sequences of trees of the same order such that $B \LE B'$. Then we have
\[W(\Gr(B)) \geq W(\Gr(B')).\] \label{majotree}
\end{theorem}


It is natural to expect a similar ``greedy" structure for unicyclic graphs that minimize the Wiener index, especially since we know the concavity is respected for certain paths in the graph. The main difference is for the case of unicyclic graphs, the optimal graph may not be unique. Figure \ref{fig:nonisograph} shows two different graphs with the same degree sequence and girth that both minimize the Wiener index.  

\begin{figure}[H]
 \centering
  \begin{subfigure}[b]{0.4\linewidth}
    \includegraphics[width=\linewidth]{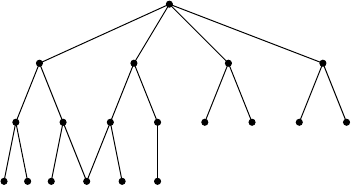}
  \end{subfigure}
  \begin{subfigure}[b]{0.4\linewidth}
    \includegraphics[width=\linewidth]{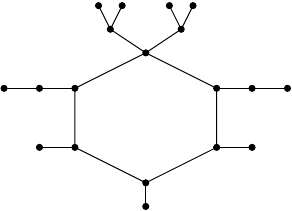}
  \end{subfigure}
  \caption{Two non-isomorphic graphs with minimum Wiener index among all unicyclic graphs of fixed $\G=6$ and given degree sequence $(4,3,3,3,3,3,3,3,2,2,1,\dots,1)$}
  \label{fig:nonisograph}
\end{figure}

\section{Structure of the optimal graph for the vertices along the cycle}\label{cycle}

Let $G $ be a unicyclic graph and $C_{\G}=v_1v_2\dots v_{\G}v_1$ be its cycle. After removing the edges of $C_{\G}$, we are left with trees $T_1,\dots,T_{\G}$ attached to each $v_i$. The unicyclic graph $G$ will then be denoted by $C_{\G}(T_1,\dots,T_{\G})$. In this section we will focus on the vertices $v_i$'s on the cycle, namely by arranging them according to the order of the subtrees attached to each vertex.

In \cite{Du2010}, the authors proposed the following formula to compute the Wiener index of a unicyclic graph:

\begin{align}\label{EqWiener}
W(G)
&=\left(n-\frac{\G}{2}\right)\left\lfloor\frac{\G^2}{4} \right\rfloor+(\G-1)\sum_{i=1}^{\G}\A_i+\sum_{i=1}^{\G}W(T_i)\\&+\sum_{i=1}^{\G-1}\sum_{j=i+1}^{\G} \ell_i\A_j+\ell_i\ell_jd_{C_{\G}}(v_i,v_j)+\ell_j\A_i  \nonumber,
\end{align}
where $\A_i=D_{T_i}(v_i)$ and $\ell_i=|V(T_i)|-1$. 

For convenience, we call a unicyclic graph \textsf{``}optimal\textsf{''} if it minimizes the Wiener index among all unicyclic graphs in $\Ud^{\G}$.
In order to understand the structure of the optimal unicyclic graph, we consider a special labelling of the vertices of the cycle $C_{\G}$, named ``$(x_0,y_0)$-labelling", according to two distinct vertices $x_0,y_0$ of $C_{\G}$. Let $P_{x_0y_0}$ be the shortest path from $x_0$ to $y_0$ and $\di =d_{C_{\G}}(x_0,y_0)$. We label the vertices of $C_{\G}$ as follows:
\begin{align*}
x_i\quad&\text{if}\quad v_i \in P_{x_0y_0}\quad\text{such that}\quad d(v_i,x_0)=i<d(v_i,y_0),\quad\text{for}\quad i=1,\dots,k_1,\\
x'_i\quad&\text{if}\quad v_i \not \in P_{x_0y_0}\quad \text{such that}\quad  d(v_i,x_0)=i<d(v_i,y_0),\quad\text{for}\quad i=1,\dots,k_2,\\
y_i\quad&\text{if}\quad v_i \in P_{x_0y_0}\quad\text{such that}\quad d(v_i,y_0)=i<d(v_i,x_0),\quad\text{for}\quad i=1,\dots,k_1,\\
y'_i\quad&\text{if}\quad v_i \not \in P_{x_0y_0}\quad \text{such that}\quad d(v_i,y_0)=i<d(v_i,x_0),\quad\text{for}\quad i=1,\dots,k_2,\\
z\quad&\text{if}\quad v_i \in P_{x_0y_0}\quad \text{such that}\quad d(v_i,x_0)=d(v_i,y_0)=\frac{\di}{2},\\
z'\quad&\text{if}\quad v_i \not \in P_{x_0y_0}\quad \text{such that}\quad d(v_i,x_0)=d(v_i,y_0)=\frac{\G-\di}{2},
\end{align*}
where $k_1=\left\lfloor\frac{\di-1}{2}\right\rfloor$ and $k_2=\left\lfloor\frac{\G-\di-1}{2}\right\rfloor$.
Note that $x_i,z,z'$ may not exist and $z$(resp. $z'$) only appear if $\di$(resp. $\G-\di$) are even.
Let $X_i,X'_i,Y_i,Y'_i,Z,Z'$ denote the component that contains the corresponding vertex. In addition, we denote $X'_{>}=\cup_{i=1}^{\left\lfloor\frac{\G}{2}\right\rfloor-\di}X'_i$ and $Y'_{>}= \cup_{i=1}^{\left\lfloor\frac{\G}{2}\right\rfloor-\di}Y'_i$. An example of $(x_0,y_0)$-labelling with $z$ and $z'$ can be seen in Figure \ref{cyclex0y0}.

\begin{figure}[H]
\centering
\begin{tikzpicture}[scale=1]
\draw [-] (0,0) -- (1,0)--(2,-1);
\draw [-](2,-1)--(2,-3)--(1,-4); 
\draw [-] (0,0) -- (-1,-1);
\draw [-](-1,-1) -- (-1,-3)--(0,-4)--(1,-4) ;
\draw [-] (0,0) -- (0.4,0.75)-- (-0.4,0.75)--(0,0) ;
\draw [-] (1,0) -- (1.4,0.75)-- (1-0.4,0.75)--(1,0) ;
\draw [-](-1,-1)--(-1.75,-1.4)--(-1.75,-0.6)--(-1,-1);
\draw [-](-1,-2)--(-1.75,-2.4)--(-1.75,-1.6)--(-1,-2);
\draw [-](-1,-3)--(-1.3,-3.8)--(-1.75,-3)--(-1,-3);
\draw [-](2,-1)--(2.3,-0.3)--(2.75,-1)--(2,-1);
\draw [-](2,-2)--(2.75,-2.4)--(2.75,-1.6)--(2,-2);
\draw [-](2,-3)--(2.75,-3.4)--(2.75,-2.6)--(2,-3);
\draw [-] (0,-4) -- (0.4,-4.75)-- (-0.4,-4.75)--(0,-4) ;
\draw [-] (1,-4) -- (1.4,-4.75)-- (1-0.4,-4.75)--(1,-4) ;
\draw (0,1) node{$X_{0}$};
\draw (1,1) node{$X_{1}$};
\draw (-2,-1) node{$X'_{1}$};
\draw (-1.8,-3.5) node{$Z'$};
\draw (-2.2,-2) node{$X'_2$};
\draw (3,-0.6) node{$Z$};
\draw (3.2,-2) node{$Y_{1}$};
\draw (3.2,-3) node{$Y_{0}$};
\draw (1,-5) node{$Y'_{1}$};
\draw (0,-5) node{$Y'_{2}$};
\fill (0,0) circle(0.05) node [below] {\small $x_{0}$};
\fill (1,0) circle(0.05) node [below] {\small $x_{1}$};
\fill (-1,-1) circle(0.05) node [right] {\small $x'_{1}$};
\fill (-1,-2) circle(0.05) node [right] {\small $x'_{2}$};
\fill (-1,-3) circle(0.05) node [right] {\small $z'$};
\fill (0,-4) circle(0.05) node [above] {\small $y'_{2}$};
\fill (2,-1) circle(0.05) node [left] {\small $z$};
\fill (2,-2) circle(0.05) node [left] {\small $y_{1}$};
\fill (2,-3) circle(0.05) node [left] {\small $y_{0}$};
\fill (1,-4) circle(0.05) node [above] {\small $y'_{1}$};

\end{tikzpicture}		
\caption{$(x_0,y_0)$-labelling}
\label{cyclex0y0}
\end{figure}
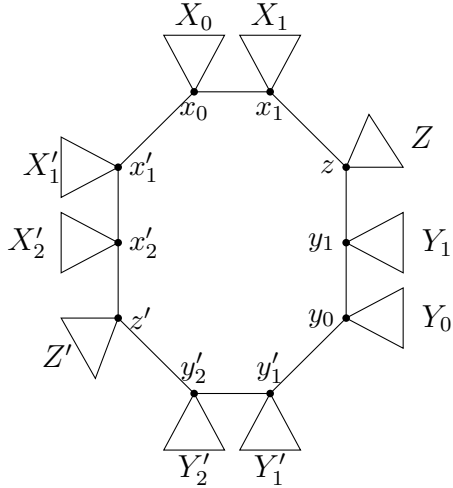

\begin{lem}\label{lemx0y0}
In the $(x_0,y_0)$-labelling, if $|V(X_i)|\geq |V(Y_i)|$ for $i=1,\dots,k_1$, $|V(X'_{>})| \geq |V(Y'_{>})|$, and $|V(X'_i)|\geq |V(Y'_i)|$ for $i=\left\lfloor\frac{\G}{2}\right\rfloor-\di+1,\dots,k_2$, then we can assume 
\begin{equation}\label{eqx0y0}
|V(X_0)|\geq |V(Y_0)|
\end{equation}in an optimal unicyclic graph.
\end{lem}

\begin{proof}
Suppose for contradiction that \eqref{eqx0y0} does not hold. We will prove that switching $X_0$ and $Y_0$ will not increase the Wiener index.
In this operation, the distances between vertices only change if exactly one of the end vertices is in $V(X_0)\cup V(Y_0)$. 

However, for a labelling with $z$ (resp. $z'$), the distance with one end vertex in $Z$ (resp. $Z'$) and a vertex in $V(X_0)\cup V(Y_0)$ does not change. 

For any vertex in $X_i$ $(i=1,\dots,k_1)$ and any vertex in $X_0$ (resp. $Y_0$), this operation increases (resp. decreases) the distance by $\di-2i$. On the other hand, for any vertex in $Y_i$ $(i=1,\dots,k_1)$ and any vertex in $X_0$ (resp. $Y_0$), the distance decreases (resp. increases) by the same amount $\di-2i$.  

Similarly, the same holds for any vertex in $X'_{>}$ and any vertex in $X_0$ (resp. $Y_0$) and for any vertex in $Y'_{>}$  and any vertex in $X_0$ (resp. $Y_0$), with the amount of change which is now instead $\di$.

Finally, for  any vertex in $X'_i$ $(i=\left\lfloor\frac{\G}{2}\right\rfloor-\di+1 ,\dots,k_2)$ and any vertex in $X_0$ (resp. $Y_0$ and for any vertex in $Y'_i$ $(i=\left\lfloor\frac{\G}{2}\right\rfloor-\di+1 ,\dots,k_2)$ and any vertex in $X_0$ (resp. $Y_0$), we follow the same reasoning with the amount of change $\G-\di-2i$.


Thus the total amount of changes that occur is
\begin{align*}
&\sum_{i=1}^{k_1} (\di-2i)(|V(X_i)|-|V(Y_i)|)(|V(X_0)|-|V(Y_0)|)\\
&+\di (|V(X'_{>})|-|V(Y'_{>})|)(|V(X_0)|-|V(Y_0)|)\\
&+\sum_{i=\left\lfloor\frac{\G}{2}\right\rfloor-\di+1}^{k_2}(\G-\di-2i) (|V(X'_i)|-|V(Y'_i)|(|V(X_0)|-|V(Y_0)|)\leq 0.
\end{align*}

Note that the equality holds if and only if $|V(X_i)|=|V(Y_i)|$ for all $i$, $|V(X'_i)|=|V(Y'_i)|$ for all $i$, and $|V(X'_{>})|=|V(Y'_{>})|$. However, in this case, switching $X_0$ and $Y_0$ results in the same graph. 
\end{proof}

\begin{lem}\label{lemx2y2}
In the $(x_0,y_0)$-labelling, if $|V(X_i)|\geq |V(Y_i)|$ for $i=0,1,\dots,k_1$, and $|V(X'_i)|\geq |V(Y'_i)|$ for $i=\left\lfloor\frac{\G}{2}\right\rfloor-\di+1,\dots,k_2$, then we can assume 
\begin{equation}\label{eqx2y2}
|V(X'_{>})| \geq |V(Y'_{>})|
\end{equation}in an optimal unicyclic graph.
\end{lem}

\begin{proof}
Suppose for contradiction that \eqref{eqx2y2} does not hold. We will prove that switching $X'_{>}$ and $Y'_{>}$ will not increase the Wiener index.
In this operation, the distances between vertices only change if exactly one of the end vertices is in $V(X'_{>})\cup V(Y'_{>})$. Using the same reasoning as in Lemma \ref{lemx0y0}, the total amount of changes is computed as follows:
\begin{align*}
&\sum_{i=0}^{k_1} (\di-2i)(|V(X_i)|-|V(Y_i)|)(|V(X'_>)|-|V(Y'_>)|)\\
&+\sum_{i=\left\lfloor\frac{\G}{2}\right\rfloor-\di+1}^{k_2}(\G-\di-2i) (|V(X'_i)|-|V(Y'_i)|(|V(X'_>)|-|V(Y'_>)|)< 0.
\end{align*}
\end{proof}

\begin{lem}\label{lemx1y1}
In the $(x_0,y_0)$-labelling, if $d_{C_{\G}}(x_0,y_0)=1,2$, then
 \begin{equation*}
 |V(X_0)|\geq|V(Y_0)|\Longleftrightarrow|V(X'_>)|\geq |V(Y'_{>})|,
 \end{equation*}
in an optimal unicyclic graph.
\end{lem}

\begin{proof}
Note that in this situation, there is no $x_i$ and $k_2=\left\lfloor\frac{\G-\di-1}{2}\right\rfloor=\left\lfloor\frac{\G}{2}\right\rfloor-\di $. Similar to the proof of Lemma \ref{lemx0y0}, the total amount of changes in the Wiener index, by switching $X_0$ and $Y_0$, that occur is
\begin{align*}
& \di(|V(X'_{>})|-|V(Y'_{>})|)(|V(X_0)|-|V(Y_0)|),
\end{align*}
which is negative if and only if $|V(X'_{>})|-|V(Y'_{>})|$ and $|V(X_0)|-|V(Y_0)|$ have different signs.
\end{proof}

\begin{corollary} \label{corx0y0}
In the $(x_0,y_0)$-labelling, if $|V(X_i)|\geq |V(Y_i)|$ for $i=1,\dots,k_1$, $|V(X'_{>})|\geq |V(Y'_{>})|$ and $|V(X'_i)|\geq |V(Y'_i)|$ for $i=\left\lfloor\frac{\G}{2}\right\rfloor-\di+1,\dots,k_2$, then we can assume 
$d(x_0)\geq d(y_0)$ in an optimal unicyclic graph.
\end{corollary}

\begin{proof}
Let $G$ be an optimal unicyclic graph.
Suppose (for contradiction) that $d(x_0)<d(y_0)$. Set $s=d(y_0)-d(x_0)$ and let $v_1,v_2,\dots,v_s$ be neighbours of $y_0$ other than $y_{1}$ and $y'_{1}$. We perform the following operation: remove the $s$ edges $y_0v_h$ and add edges $x_0v_h$ $(h=1,\dots,s)$ instead. We will show that this operation will not increase the Wiener index. Note that the degree sequence is preserved and then $d(x_0) \geq d(y_0)$. Let $V_h$ be the subtree attached to $v_h$ after the removal of $y_0v_h$. Let $S=\cup_{h=1}^sV_h$. The distances between vertices only change if exactly one of the end vertices is in $S$.
Similar to the previous proof, the total change of the Wiener index in this operation is

\begin{align*}
&\sum_{i=1}^{k_1} (\di-2i)(|V(Y_i)|-|V(X_i)|)|S|+\di (|V(Y'_{>})|-|V(X'_{>})|)|S|\\
&+\sum_{i=\left\lfloor\frac{\G}{2}\right\rfloor-\di+1}^{k_2} (\G-\di-2i) (|V(Y'_i)|-|V(X'_i)||S|)+\di|S|(|V(Y_0)|-|S|-|V(X_0)|)\\& < 0.
\end{align*}
Note that since $|S|\geq 1$, the last term is always negative.
\end{proof}


For a cycle in an optimal unicyclic graph, we can label the vertices of $C_{\G}$ as $w_1,w_2,\dots$ and $u_1,u_2,\dots$ and the components as $W_i$ and $U_i$, while $U_1$ is the component with most vertices (Figure \ref{fig:cycleC}) such that the following holds:

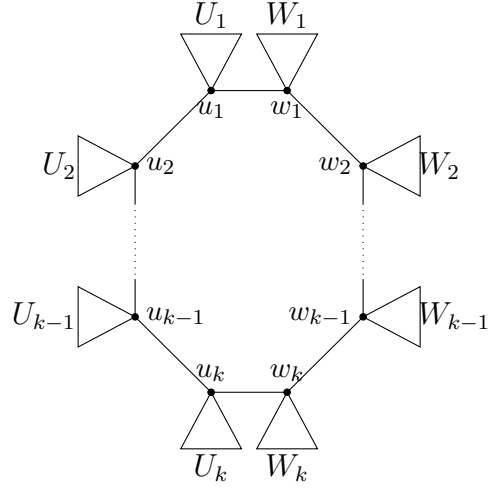
\begin{figure}
\centering
\begin{tikzpicture}[scale=1]
\draw [-] (0,0) -- (1,0)--(2,-1) -- (2,-1.5);
\draw [-](2,-2.5)--(2,-3)--(1,-4); 
\draw [-] (0,0) -- (-1,-1)--(-1,-1.5);
\draw [-](-1,-2.5) -- (-1,-3)--(0,-4)--(1,-4) ;
\draw [dotted](2,-1.5)--(2,-2.5);
\draw [dotted](-1,-1.5)--(-1,-2.5);
\draw [-] (0,0) -- (0.4,0.75)-- (-0.4,0.75)--(0,0) ;
\draw [-] (1,0) -- (1.4,0.75)-- (1-0.4,0.75)--(1,0) ;
\draw [-](-1,-1)--(-1.75,-1.4)--(-1.75,-0.6)--(-1,-1);
\draw [-](-1,-3)--(-1.75,-3.4)--(-1.75,-2.6)--(-1,-3);
\draw [-](2,-1)--(2.75,-1.4)--(2.75,-0.6)--(2,-1);
\draw [-](2,-3)--(2.75,-3.4)--(2.75,-2.6)--(2,-3);
\draw [-] (0,-4) -- (0.4,-4.75)-- (-0.4,-4.75)--(0,-4) ;
\draw [-] (1,-4) -- (1.4,-4.75)-- (1-0.4,-4.75)--(1,-4) ;
\draw (0,1) node{$U_{1}$};
\draw (1,1) node{$W_{1}$};
\draw (-2,-1) node{$U_{2}$};
\draw (-2.2,-3) node{$U_{k-1}$};
\draw (3,-1) node{$W_2$};
\draw (3.2,-3) node{$W_{k-1}$};
\draw (1,-5) node{$W_{k}$};
\draw (0,-5) node{$U_{k}$};
\fill (0,0) circle(0.05) node [below] {\small $u_{1}$};
\fill (1,0) circle(0.05) node [below] {\small $w_{1}$};
\fill (-1,-1) circle(0.05) node [right] {\small $u_{2}$};
\fill (-1,-3) circle(0.05) node [right] {\small $u_{k-1}$};
\fill (0,-4) circle(0.05) node [above] {\small $u_{k}$};
\fill (2,-1) circle(0.05) node [left] {\small $w_{2}$};
\fill (2,-3) circle(0.05) node [left] {\small $w_{k-1}$};
\fill (1,-4) circle(0.05) node [above] {\small $w_{k}$};
\end{tikzpicture}		
\caption{Labelling of $C_{\G}$ with an even $\G$.}
\label{fig:cycleC}
\end{figure}

\begin{theorem} \label{thm:cycle}
In an optimal unicyclic graph labelled as in Figure \ref{fig:cycleC}, the following holds:
\begin{equation*}
|V(U_1)| \geq |V(W_1)| \geq |V(U_2)| \geq |V(W_2)|\geq \cdots \geq |V(U_k)|\geq |V(W_k)|,
\end{equation*}
if $\G=2k$; and
\begin{equation*}
|V(U_1)| \geq |V(W_1)| \geq |V(U_2)| \geq |V(W_2)|\geq \cdots \geq |V(U_{k-1})|\geq |V(W_{k-1})|\geq |V(U_{k})|,
\end{equation*}
if $\G=2k-1$.
\end{theorem}

\begin{proof}
We only show the proof for an even $\G=2k$, the odd case is similar. We can assume that $|V(U_1)| \geq |V(W_1)| \geq |V(U_2)|$, where $U_1$ is the branch with the most vertices. In fact, from \eqref{EqWiener}, switching the branches $T_i$'s will only  affect the sum $\sum_{i=1}^{\G-1}\sum_{j=i+1}^{\G} \ell_i\ell_jd_{C_{\G}}(u_i,u_j)$, which means branches with the most number of vertices must be closer to each other. 

\vspace{1cm}
\textbf{Claim 1:} If $|V(U_1)| \geq |V(W_1)|$, then $|V(U_k)| \geq |V(W_k)|$.
\begin{proof}[Proof of Claim 1]
Suppose it is not true, i.e. $|V(U_k)| < |V(W_k)|$. 
By Lemma \ref{lemx1y1} in taking $x_0=u_k$ and $y_0=w_k$, we have 
\begin{align*}
|X'_{>}|=\sum_{i=1}^{k-1}|V(U_i)|&\leq \sum_{i=1}^{k-1}|V(W_i)|=|Y'_{>}|\\
\sum_{i=1}^k|V(U_i)|&<\sum_{i=1}^k|V(W_i)|.
\end{align*} 
However, by the same Lemma in considering $x_0=u_1$ and $y_0=w_1$, we get 
\begin{align*}
\sum_{i=2}^{k}|V(U_i)|&\geq \sum_{i=2}^{k}|V(W_i)|\\
\sum_{i=1}^{k}|V(U_i)|&\geq \sum_{i=1}^{k}|V(W_i)|\\
\end{align*} which yields a contradiction.
\end{proof}

\textbf{Claim 2:} If $|V(U_1)| \geq |V(W_1)|$, then $|V(U_2)|\geq |V(W_2)|$. 
\begin{proof}[Proof of Claim 2]
Suppose for contradiction that $|V(U_2)|<|V(W_2)|$. We may assume that $\sum_{i=2}^{k-1}|V(U_i)|\geq \sum_{i=2}^{k-1}|V(W_i)|$. In fact, if $\sum_{i=2}^{k-1}|V(U_i)|< \sum_{i=2}^{k-1}|V(W_i)|$, from Lemma \ref{lemx2y2}, switching $\cup_{i=2}^{k-1}U_i$ with $\cup_{i=2}^{k-1}W_i$ results in a graph with a smaller Wiener index, which contradicts the optimality of the graph. Hence
\begin{align*}
\sum_{i=2}^{k-1}|V(U_i)|&\geq \sum_{i=2}^{k-1}|V(W_i)|\\
|V(U_2)|+\sum_{i=3}^{k-1}|V(U_i)|&\geq \sum_{i=3}^{k-1}|V(W_i)|+|V(W_2)|\\
\sum_{i=3}^{k-1}|V(U_i)|&\geq \sum_{i=3}^{k-1}|V(W_i)|+\underbrace{|V(W_2)|-|V(U_2)|}_{>0}> \sum_{i=3}^{k-1}|V(W_i)|,
\end{align*}
then, we can use Lemma \ref{lemx0y0}, by taking $x_0=u_2$ and $y_0=w_2$. All the required conditions are satisfied since $|V(U_1)|\geq |V(W_1)|$, $V(U_k)\geq |V(W_k)|$, and $\sum_{3}^{k}|V(U_i)|\geq \sum_{3}^{k}|V(W_i)|$. Thus, $|V(U_2)|\geq |V(W_2)|$ which yields a contradiction.
\end{proof}

We use the same approach as in Claim 2, to prove that if $|V(U_k)| \geq |V(W_k)|$, then $|V(U_{k-1})|\geq |V(W_{k-1})|$. We iterate the process for the next $i$ and $k-i+1$ to finally get $|V(U_{i})|\geq |V(W_i)|$, for $i=1,\dots, k$.

Moreover, using the same pattern as in Claim 1, we prove that if $|V(W_1)|\geq |V(U_2)|$, then $|V(W_{k-1})|\geq |V(U_{k})|$ (Lemma \ref{lemx1y1}). Then, we may perform as in Claim 2 to show that $|V(W_2)|\geq |V(U_3)|$ and symmetrically $|V(W_{k-2})|\geq |V(U_{k-3})|$. Indeed, all the conditions in Lemma \ref{lemx0y0} are satisfied, for $x_0=w_2$ and $y_0=u_3$ as well as for $x_0=w_{k-2}$ and $y_0=u_{k-3}$. By iterating the process for the next $i$ and $k-i+1$, the following $|V(W_{i-1})|\geq |V(U_i)|$ holds for $i=2,\dots,k$.
\end{proof}

\begin{corollary} \label{Cor:deg}
In an optimal unicyclic graph, for a cycle with labelling as in Theorem \ref{thm:cycle}, we have
\begin{equation*}
d(u_1) \geq d(w_1) \geq d(u_2) \geq d(w_2)\geq \cdots \geq d(u_k)\geq d(v_k)
\end{equation*}
if $\G$ is even; and
\begin{equation*}
d(u_1) \geq d(w_1) \geq d(u_2) \geq d(w_2)\geq \cdots \geq d(u_{k-1})\geq d(w_{k-1})\geq d(u_{k})
\end{equation*}
if $\G$ is odd.
\end{corollary}

\begin{proof}
We are going to only prove the even case as the odd case can be shown with a similar argument. From Theorem \ref{thm:cycle}, we have 
\begin{equation*}
|V(U_1)| \geq |V(W_1)| \geq |V(U_2)| \geq |V(W_2)| \geq \cdots \geq |V(U_k)|=|V(W_k)|.
\end{equation*}

\textbf{Case 1:} Let us consider a $(x_0,y_0)$-labelling, where $x_0=u_i$ and $y_0=w_i$. Since $|V(U_i)|\geq |V(W_i)|$ for all $i$, by Corollary \ref{corx0y0} $d(u_{i})=d(x_i)\geq d(y_i)=d(v_{i})$ for $i=1,\dots,k$.

\textbf{Case 2:} Next we choose  $x_0=w_i$ and $y_0=u_{i+1}$, for $i=1,\dots,k-1$. Since, $|V(W_i)| \geq |V(U_{i+1})|$, by Corollary \ref{corx0y0}, we again obtain $d(w_i) \geq d(u_{i+1})$.

The result follows by combining both cases. 
\end{proof}
The following proposition is directly deduced from Theorem \ref{thm:cycle}.
 
\begin{corollary}\label{prop}
Let $G \in \Ud^{\G}$ with $G=C_{\G}(T_1,\dots,T_{\G})$ such that $|V(T_1)|\geq |V(T_2)|\geq \cdots \geq |V(T_{\G})|$.
 We denote by $ord(G)$ the unicyclic graph defined as follows:
\begin{equation*}
ord(G)=C_{\G}(T_1,T_3,\dots T_{\lfloor \frac{\G}{2}\rfloor},T_{\lfloor \frac{\G}{2}\rfloor-1},\dots,T_4,T_{2}),
\end{equation*}then $W(G)\geq W(ord(G))$.
 \end{corollary}  

From now on, we assume our optimal unicyclic graphs are ordered according to Corollary \ref{prop} and we follow the labelling of the cycle as in Theorem \ref{thm:cycle}.

Moreover, the following theorem from Hardy, Littlewood and P\'{o}lya \cite{hardy} that connects non-negative sequences and bilinear forms will play a central role in the upcoming lemmas. For a sequence $a$ of the form $a_{-n},\dots a_{-n+1}, \dots, a_{-1},a_0$ $,a_1,\dots$ $ a_{n-1},a_n$, let $a^+$ be the sequence obtained by rearranging the elements so that
\begin{equation*}
	a^+_0 \geq a^+_1 \geq a^+_{-1} \geq a^+_{2} \geq a^+_{-2} \geq \cdots \geq a^+_{n} \geq a^+_{-n}.
\end{equation*}

\begin{theorem}[{\cite[Theorem 371]{hardy}}]\label{TheoremHardy}
	Suppose that $c, x$, and $y$ are non-negative sequences, and $c$ is symmetrically decreasing, so that 
	\begin{equation*}
		c_0 \geq c_1 =c_{-1}\geq c_2 =c_{-2} \geq \cdots \geq c_{2k}=c_{-2k},
	\end{equation*}
	while the $x$ and $y$ are given except for their order. Then the bilinear
	form
	\[\sum_{r=-k}^k \sum_{s=-k}^kc_{r-s}x_ry_s\]
	attains its maximum when $x$ is $x^+$ and $y$ is $y^+$.
	\label{thm.hardy}
\end{theorem}

\begin{lemma}\label{remedge}
Let $G \in \Ud^{\G}$, be optimal. Let $T_1$ be the tree obtained from $G$ by removing the edge $e_1=u_kw_{k-1}$ (resp. $u_kw_k$) and $T_2$ a tree obtained by removing any other edge $e_2$ on $C_{\G}$. Then, $W(T_1)\leq W(T_2)$.
\end{lemma}

\begin{proof}
The proof basically follows the argument found in \cite{Wang2009}. The Wiener index of a tree can be written as $W(T)=\sum_{uv \in E(T)}n_{uv}(T)n_{vu}(T)$, where $n_{uv}(T)$ is the number of vertices closer to $u$ after removing $uv$ and $n_{vu}(T)$ the number of vertices closer to $v$. We may notice that the contribution of any edges not on $C_{\G}\backslash \{e_1\}$ or $C_{\G}\backslash \{e_2\}$ remains the same for both cases. The only difference is then for the edges along $C_{\G}\backslash \{e_1\}$ and $C_{\G}\backslash \{e_2\}$. We will end up with a bilinear form and from Theorem \ref{TheoremHardy}, the contribution on the Wiener index of $T_1$ from $C_{\G}\backslash \{e_1\}$ is smaller than the contribution of the Wiener index of $T_2$ from $C_{\G}\backslash \{e_2\}$.
\end{proof}

\begin{rem}\label{lowdeg}
Note that if we remove the edge $u_kw_{k-1}$ (resp. $u_kw_k$), we obtain a path which also satisfies the concavity in Lemma \ref{optimalpath}.

In particular, if the length of the cycle differs at most by one from the other maximal paths in the new obtained tree, then Remark \ref{height} is satisfied and the new tree is greedy.
\end{rem}

\section{Difference between Wiener indices of trees and unicyclic graphs}

\label{difference}



Let $G \in \Ud^{\G}$ and label its cycle by $C_{\G}=u_1u_2\dots u_k(w_k)w_{k-1}\dots w_2w_1$. 
Let $e$ be an edge of $C_{\G}$. From Lemma \ref{remedge}, we may assume $e=u_kw_{k-1}$ (resp. $u_kw_k$). After the removal of the edge $e$ we are left with a tree $T_G$. Let us compute the difference between the Wiener indices of the unicyclic graph $G$ and the corresponding tree $T_G$.

\begin{lem}\label{lem:T+uv}
Let $G$ be any unicyclic graph with cycle $C_{\G}=u_1u_2\dots u_k(w_k) $ $w_{k-1}\dots w_1u_1$ and $e=u_kw_{k-1}$ (resp. $u_kw_k$). Let $T_G$ be the tree obtained after the removal of the edge $e$. We set $\Di(G)=W(T_G)-W(G)$. Then, for a cycle of even order $2k$:
\begin{align*}
\Di(G)=\sum_{i=2}^{k}\sum_{j=k+2-i}^{k}2(i+j-k-1)|V(U_i)||V(W_j)|,
\end{align*}	
and for a cycle of odd order $2k-1$:	
\begin{align*}
\Di(G)=\sum_{i=2}^{k}\sum_{j=k+1-i}^{k-1}(2(i+j-k)-1)|V(U_i)||V(W_j)|.
\end{align*}
\end{lem}

\begin{proof}
For a pair of vertices $(x,y)$ in $G$, the distances in the unicyclic graph and the new obtained tree remain the same, i.e., $d_{G}(x,y)=d_{T_G}(x,y)$ unless we use the edge $e$ as a ``shortcut". In that case $d_{T_G}(x,y)>d_{G}(x,y)=d_{T_G}(x,u_k)+1+d_{T_G}(w_{k-1},y)$ for an even cycle and $d_{T_G}(x,y)>d_{G}(x,y)=d_{T_G}(x,u_{k})+1+d_{T_G}(w_k,y)$ for an odd cycle. The latter happens if and only if $x \in U_i$ and $y \in W_j$ such that $d_T(u_i,w_j)>k$ and the difference is exactly:
\begin{align*}
d_{T_G}(x,y)-d_G(x,y)&=d_{T_G}(u_i,w_j)-d_G(u_i,w_j)\\
&=d_{T_G}(u_i,u_1)+1+d_{T_G}(w_1,v_j)\\&-d_{T_G}(u_i,u_k)-d_{T_G}(v_j,v_k)-1\\
&=i-1+1+j-1-k+i-k+j-1\\
&=2(i+j-k-1),
\end{align*} for an even cycle $2k$ and
\begin{align*}
d_{T_G}(x,y)-d_G(x,y)&=d_{T_G}(u_i,w_j)-d_G(u_i,w_j)\\
&=d_{T_G}(u_i,u_1)+1+d_{T_G}(w_1,w_j)\\&-d_{T_G}(u_i,u_{k})-d_{T_G}(w_j,w_{k-1})-1\\
&=i-1+1+j-1-k+i-k+j\\
&=2(i+j-k)-1,
\end{align*} 
for an odd cycle $2k-1$.
Hence, the statement holds.
\end{proof}


\begin{lem}\label{lem:diff}
Let $G \in \Ud^{\G}$ with $G=C_{\G}(U_1,\dots,U_k,(W_{k}),W_{k-1},\dots,W_1,U_1)$ and $e=u_kw_{k-1}$ (resp. $u_kw_k$). The difference $\Di(G)$ is maximized if the subtrees along the cycle respect the following arrangement:
\begin{equation*}
|V(U_1)| \leq |V(W_1)| \leq |V(U_2)| \leq |V(W_2)| \leq \cdots \leq |V(U_k)|\leq|V(W_k)|,
\end{equation*}
for an even cycle $2k$ and
\begin{equation*}
|V(U_1)| \leq |V(W_1)| \leq |V(U_2)| \leq |V(V_2)| \leq \cdots \leq |V(U_{k-1})|\leq|V(W_{k-1})|\leq |V(U_{k})|,
\end{equation*}
for an odd cycle $2k-1$.
\label{lem:diff}
 \end{lem}

\begin{proof}
We denote by $A$ the sequence whose indices are from $-k$ to $k-1$ (resp. $k$) such that 
\begin{align*}A_{-k}&=|V(U_{1})|,A_{-k+1}=|V(U_{2})|, \dots, A_{-1}=|V(U_k)|,\\A_0&=|V(W_{k})|,A_1=|V(W_{k-1})|,A_{2}=|V(W_{k-2})|,\dots ,A_{k-1}=|V(W_{1})|,\end{align*} 
for a cycle of even order $2k$ and
\begin{align*}A_{-k+1}&=|V(U_{1})|,A_{-k+1}=|V(U_{2})|, \dots, A_{-1}=|V(U_{k-1})|,\\A_0&=|V(U_{k})|,A_1=|V(W_{k-1})|,A_{2}=|V(W_{k-2})|,\dots ,A_{k-1}=|V(W_1)|,\end{align*}
for a cycle of odd order $2k-1$. 
We provide only the proof for an even cycle, the other case can be shown in a similar manner. By a change of indices, from Lemma \ref{lem:T+uv} we obtain
\begin{align*}
\Di(G)=\sum_{r=-k+1}^{-1}\sum_{s=0}^{r+k-1}2(k-(s-r))A_rA_s.
\end{align*}	
Since $A_rA_s = A_sA_r$, we set \[d_{r-s}=d_{s-r}=\begin{cases}k-|s-r| ,&|s-r| \leq k-1\\0, & \text{otherwise}\end{cases}.\] Then, 
 \begin{align*}
 \Di(G)&=\sum_{r=-k+1}^{-1}\sum_{s=0}^{r+k-1}2(k-(s-r))A_rA_s=\sum_{r=-k+1}^{-1}\sum_{0\leq s-r\leq k-1}2(k-(s-r))A_rA_s\\
 &=\sum_{r=-k+1}^{-1}\sum_{0\leq s-r\leq k-1}(k-(s-r))A_rA_s+\sum_{s=-k+1}^{-1}\sum_{0\leq r-s\leq k-1}(k-(r-s))A_rA_s\\
 & =\sum_{r=-k}^{k-1}\sum_{s=-k}^{k-1}d_{r-s}A_rA_s,\end{align*} 
where $d_{r-s}$ is symmetrically decreasing:
\begin{align*}&d_0=k\geq d_1=d_{-1}=k-1
\geq d_2=d_{-2}=k-2\\&\geq\cdots\geq d_{k-1}=d_{-k+1}=1\geq d_{k}=d_{-k}= \cdots d_{2k-1}=d_{-2k+1}=0.\end{align*}
Hence, by using Theorem \ref{thm.hardy}, $\Di$ attains its maximum if $A$ is $A^+$. Therefore,
\[|V(U_1)|\leq |V(W_1)| \leq |V(W_2)| \leq \cdots \leq |V(U_k)|=|V(W_k)|.\]
\end{proof}

\begin{rem}\label{rem:diff}
Note that the maximum difference between the Wiener indices of $T_G$ and $G$ has to satisfy both inequalities of Theorem \ref{thm:cycle} and Lemma \ref{lem:diff} which oppose one another. The best case scenario for the difference to be maximal is then having the same size of trees attached to each vertex of the cycle,  namely:
 \begin{equation*}
 |V(U_1)|=|V(W_1)| =|V(U_2)| = \cdots = |V(U_k)|=|V(W_k)|.
 \end{equation*}
\end{rem}

\begin{rem}\label{bigbranch}
Let $C_{\G}=u_1u_2\dots u_k(w_k)w_{k-1}\dots w_1 u_1$. Lemma \ref{lem:diff} tells us that moving a branch from $U_j$ (resp. $W_j$) to $U_i$ (resp. $W_i$) such that $i>j$ will increase $\Di(G)$. In other words, if we remove the edge $u_kw_k$ (resp $u_kw_{k-1}$) and fix the root of $T_G$ to be $u_1$, the branches on the cycle will be greater than any other branches of the same level.
\end{rem}

\section{Structure of the optimal graph for the vertices (tree components) outside the cycle} \label{outside}

Let $G \in \Ud^{\G}$ and $C_{\G}=u_1 u_2,\dots,u_k(w_k) w_{k-1}\dots w_2 w_1 u_1$ such that $u_1$ is the vertex with highest degree among the $u_i$'s and the inequalities in Theorem \ref{thm:cycle} is satisfied. We denote by $T_G$ the tree obtained from the removal of an edge from $C_{\G}$. 

From section \ref{difference}, we have
 
 \begin{equation} \label{eq:diff}
 W(G)=W(T_G)-\Di(G).
 \end{equation}

Following Remark \ref{rem:diff} we can see that $W(T_G)$ and $\Di(G)$ are entangled. Any transformation that will increase one, will do the same to the other. However, we require $W(T_G)$ to be as small as possible and $\Di(G)$ to be as big as possible. Even though both cannot happen at the same time, we can compute the contribution of a transformation to both parameters and compare their amount. In general, the proportion of $W(T_G)$ is greater than the difference $\Di(G)$. It is natural then to first minimize $W(T_G)$ before trying to improve $\Di(G)$.   

\subsection{Minimizing $T_G$}
First, we have seen that from Lemma \ref{remedge}, the tree obtained from removing $u_kw_{k-1}$ (resp $u_kw_k$) has the minimum Wiener index among all trees that we obtain by removing any other edge on $C_{\G}$. From now on, we assume that we remove these specific edges. Let us characterize the structure of a maximal path in $T_G$ for the Wiener index to be minimum. 
\begin{claim}
Let $T_{G}$ be the tree that minimizes the Wiener index among all trees we may obtain from $G$ by removing $e$. Then, all the maximal paths in $T_G$ respect concavity centered around the vertex with highest degree in $T_G$.  
\end{claim}

\begin{proof}
From Remark \ref{atmost2vces}, all the maximal paths that involve at most one vertex from the cycle satisfy Lemma \ref{optimalpath}, and from Remark \ref{lowdeg} the path obtained from the cycle after removing $u_kw_{k-1}$ (resp $u_kw_k$) satisfies the same lemma. We are left to understand the structure of the path involving more than or equal to two vertices on the cycle and a leaf (a vertex not part of the cycle). Two cases maybe considered.

\textbf{Case 1:  The vertex with highest degree of $T_G$ is on the cycle, namely $u_1$.}
Case 1 is depicted as in Figure \ref{Fig:case5}. Let $P_{w_{k-1}u_1z_{\ell}}=w_{k-1}\dots w_1u_1 z_1 $ $z_2\dots z_{\ell}$ be such a path where $z_i \not \in C_{\G}$ for all $i$ and $\ell \leq k$. As in Section \ref{trees}, we denote $U'_1$, $Z_i$ the subtrees that contain $u_1$ and $z_i$ after the removal of the edges on the path $P_{w_{k-1}u_1z_{\ell}}$ and the edges on the cycle $C_{\G}$. 
Since $u_1$ is the vertex with highest degree of $T_G$ and from Remark \ref{bigbranch}, we need to show that \begin{equation} \label{eq:pathWZ}
|V(U_1)| \geq |V(W_1)| \geq |V(Z_1)| \geq |V(W_2)|\geq |V(Z_2)|\cdots \geq |V(W_{\ell})|\geq |V(Z_{\ell})|,
\end{equation} in particular in terms of the degrees $$d_T(u_1)\geq d_T(w_1)\geq d_T(z_1)\geq d_T(w_2) \geq d_T(z_2) \geq \cdots \geq d_T(w_\ell)\geq d_T(z_{\ell}).$$ 
\begin{figure}
\centering
\includegraphics[scale=0.7]{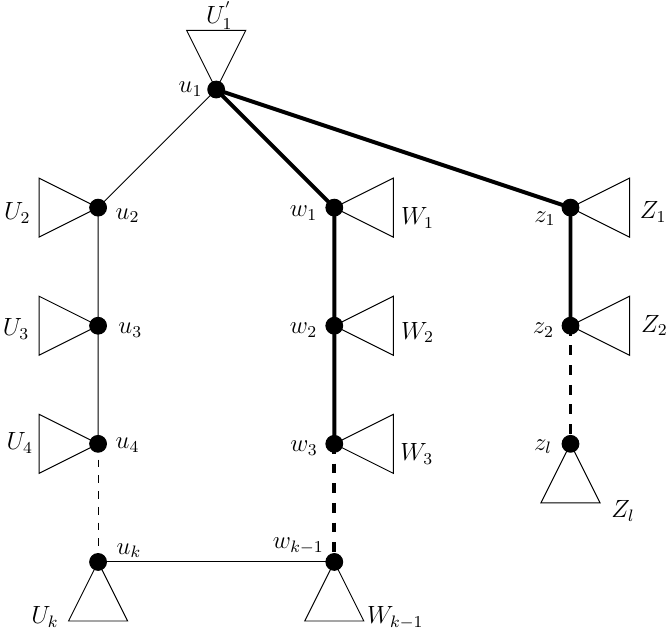}
\caption{The path $P_{w_{k-1} u_1 z_{\ell}}$.}
\label{Fig:case5}
\end{figure}
Suppose that Equation \eqref{eq:pathWZ} is not satisfied. In other words, there exist $i_0 \in \{2,\dots,k\}$ and $i_1 \in \{1,\dots,\ell\}$ such that $i_0 > i_1$ where $|V(W_{i_0})| > |V(Z_{i_1})|$ holds. We are going to switch the branches $W_{i_0}$ and $Z_{i_1}$ and prove that this operation will decrease the Wiener index.  Let $\Lo_{i_0,s}$ denote the difference between the Wiener indices of the tree after and before the above operation. Let $s=i_0-i_1$, $Z_{>h}=\cup_{i=h}^{\ell} Z_i $,$W_{>h}=\cup_{i=h}^{k} W_i $, and $U_{>h}=\cup_{i=h}^{k} U_i $.




\begin{align*}
\Lo_{i_0,s}&=(|V(Z_{i_1})|-|V(W_{i_0})|)\Big(\sum_{j=1}^{i_0-s-1}(2j+s)(|V(Z_j)|-|V(W_{j+s})|)\\&+\sum_{j=1}^{\lfloor\frac{s}{2}\rfloor}|s-2j|(|V(W_j)|-|V(W_{s-j})|)\\
&+s(|V(U'1)|+|V(U_{>2})|-|V(W_s)|)\\&+(2_{i_0}-s)(|V(Z_{>i_0-s})|-V(|W_{>i_0})|)\Big) >0.
\end{align*}
That concludes the proof of Case 1.

\textbf{Case 2:  The vertex with highest degree of $T_G$ is outside the cycle, denoted by $x$.}

\begin{figure}
\centering
\includegraphics[scale=0.7]{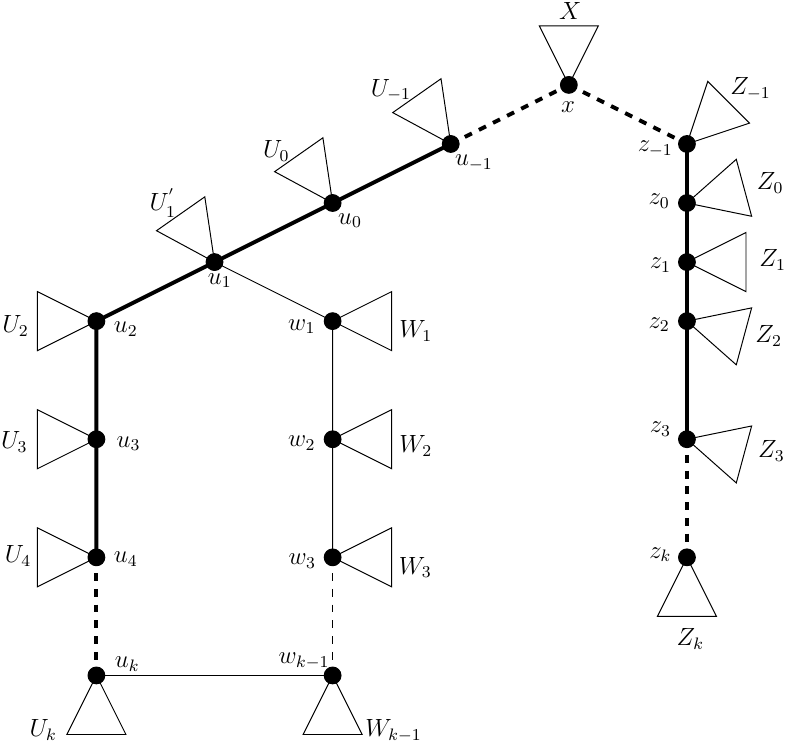}
\caption{The path $P_{u_k x z_k}$.}
\label{Fig:Section5case2}
\end{figure}
Let $P_{u_kxz_k}=u_ku_{k-1}\dots u_1u_0 u_{-1}u_{-2}\dots u_{-m} x$ $z_{-m} \dots z_{-1} z_0z_1\dots z_{k}$ be such a path where $z_i \not \in C_{\G}$ for all $i$, $u_i \not \in C_{\G}$ for $-m \leq i \leq 0$ as seen in Figure \ref{Fig:Section5case2}. We denote $U'_1$,$U_i$, $Z_i$ the subtrees that contains $u_1$, $u_i$ and $z_i$ after the removal of the edges on the path $P_{u_kxz_k}$ and the edges on the cycle $C_{\G}$. Let $Z_{>h}=\cup_{i=h}^{\ell} Z_i $,$W_{>h}=\cup_{i=h}^{k} W_i $, and $U_{>h}=\cup_{i=h}^{k} U_i $.
Since $x$ is the vertex with highest degree of $T_G$, we need to prove that \begin{align}&|V(X)| \geq |V(U_{-m})| \geq |V(Z_{-m})| \geq |V(U_{-m-1})|\geq |V(Z_{-m-1})|\cdots\label{eq:pathUZ} \\
& \geq |V(U_{1})|\geq |V(Z_{1})|\geq |V(U_2)|\geq |V(Z_2)| \cdots \geq |V(U_{k})|\geq |V(Z_{k})|\nonumber
,\end{align} more specifically in terms of the degrees 
\begin{align*}
&d_T(x)\geq d_T(u_{-m})\geq d_T(z_{-m})\geq d_T(u_{-m-1}) \geq d_T(z_{-m-1}) \geq \cdots \\
&\geq d_T(u_1)\geq d_T(z_{1})\geq d_T(u_2) \geq d_T(z_2) \cdots \geq d_T(u_{k}) \geq d_T(z_{k}).
\end{align*}
Suppose once again that Equation \eqref{eq:pathUZ} does not hold, namely there exist $i_0 \in \{2,\dots,k\}$ and $i_1 \in {1,\dots,\ell}$ such that $i_0 > i_1$ with $|V(U_{i_0})|<|V(z_{i_1})|$. We perform the switching branches as in the previous case but this time, the branches we are switching are $U_{i_0}$ and $Z_{i_1}$. 

\begin{align*}
\Lo_{i_0,s}&=(|V(Z_{i_1})|-|V(U_{i_0})|)\Big(\sum_{j=-m+s+2}^{-m+\frac{s+1}{2}}|2m+2-s|(|V(U_j)|-|V(U_{j-\frac{s+1}{2}})|)\\&+\sum_{j=-m}^{i_0-s}(2m+2+2j+s)(|V(U_{j+s})|-|V(Z_{j})|)\\
&+s(|V(U_{-m+s+1})|-|V(X)|)\\&+(2m+2+2_{i_0}-s)(|V(U_{>i_0})|-V(|Z_{>i_0-s})|)\Big)>0.
\end{align*}
That concludes the proof of Case 2 and the whole claim holds.
 \end{proof}
 
\subsection{Interplay between $\Lo_{i_0,s}$ and $\Gain_{i_0}$}
We have seen that in order to minimize $W(T_G)$, we must have Equations \ref{eq:pathWZ} for Case 1 and \ref{eq:pathUZ} for Case 2. We perfomed a switch operation between a branch along the cycle and one outside the cycle. However, the same operation will decrease $\Di(G)$ and the reverse will increase $\Di(G)$, this later is called ``gain". In other words,
let us again fix $i_0 \in \{2,\dots,k\}$ and $i_1 \in \{1,\dots,\ell\}$ such that $i_0 > i_1$. We now assume that $|V(W_{i_0})| < |V(Z_{i_1})|$ (resp. $|V(U_{i_0})| < |V(Z_{i_1})|$ ). We observe that switching back the branches $W_{i_0}$ and $Z_{i_1}$ (resp. $U_{i_0}$ and $Z_{i_1}$ ) will increase $\Di(G)$. From Lemma \ref{lem:T+uv}, we can compute the maximum gain $\Gain_{i_0}$:

\begin{align*}
\Gain_{i_0}=\sum_{j=k-i_0+2}^k 2(j-(k-i_0)-1)|V(U_j)|\left(|V(Z_{i_1})|-|V(W_{i_0})|\right), \quad \text{in Case 1}\\
\Gain_{i_0}=\sum_{j=k-i_0+2}^k 2(j-(k-i_0)-1)|V(W_j)|\left(|V(Z_{i_1})|-|V(U_{i_0})|\right), \quad \text{in Case 2}.
\end{align*}

We now can compare the amount of decrease of the Wiener index of $T_G$, called a gain and the decrease by $\Di(G)$ called the loss.

\begin{rem}\label{rem:gainloss}
If $i_0$ increases, the gain $\Gain_{i_0}$ will increase.
If $s$ increases, the loss $\Lo_{i_0,s}$ will increase.
\end{rem}

We get the following corollaries:
\begin{corollary}\label{progain}
If the loss exceeds the gain, it implies that the optimal unicyclic graph satified the concavity of all maximal paths in $T_G$. In other words, the construction of the optimal graph starts first from a greedy tree $T_G$, where we will then add the edge to vertices with lowest degrees (leaves if possible). 
\end{corollary}

\begin{corollary}\label{proloss}
If the gain exceeds the loss, this implies, all the highest degrees were given to the vertices on the cycle.
\end{corollary}


Assume that Corollary \ref{proloss} applies, namely all the highest degrees correspond to vertices on the cycle. Those vertices are already ordered as in Theorem \ref{thm:cycle} in the optimal graph. The ordering of the rest of the vertices is performed as follows:

\begin{lem}\label{prop2}
Let $G \in \Ud^{\G}$ and $G=C_{\G}(T_1,\dots,T_{\G})$. Let $T'$ be the tree obtained from $T_1,T_2,\dots,T_{\G}$ by identifying all the roots $v_i$ as a single vertex $r$ in such a way that $|V(T_1)|\geq |V(T_2)|\geq \cdots \geq |V(T_{\G})|$. 
Let $\Gr(T')$ be the greedy tree with the same degree sequence as $T'$ and $\Gr(T')$ can be decomposed as $T'_1,T'_2,\dots,T'_{\G}$ such that the degrees of the roots of $T_i$ and $T'_i$ remain the same. 
Let $G'=C_{\G}(T'_1,T'_2,\dots,T'_{\G})$ . 
Then $W(G)\geq W(G')$.

\end{lem}  

\begin{proof}
Substituting $T'$ in Equation \ref{EqWiener}, we have
\begin{align*}
W(G)=W(T')+\left(n-\frac{\G}{2}\right)\left\lfloor\frac{\G^2}{4} \right\rfloor+(\G-1)\sum_{i=1}^{\G}\A_i+\sum_{i=1}^{\G-1}\sum_{j=i+1}^{\G} \ell_i\ell_jd_{C_{\G}}(v_i,v_j).
\end{align*}

Let us denote $\A'_i=D_{\Gr(T')}(v_i)$, and $\ell_i'=|V(T_i')|-1$ where $T_i'$ are the branches in the greedy tree $\Gr(T')$. As the girth $\G$ is fixed, the term $\left(n-\frac{\G}{2}\right)\left\lfloor\frac{\G^2}{4} \right\rfloor$ is invariant under the reshuffling of the vertices. From Theorem \ref{ThmWang}, $$W(T')\geq W(\Gr(T')).$$ Moreover, $$\sum_{i=1}^{\G}\A_i \geq \sum_{i=1}^{\G}\A'_i ,$$ (note that this is true only if vertices with highest degrees are on the cycle) and the sequence $(\ell_1,\ell_2,\dots, \ell_{\G})$ is majorized by $(\ell'_1,\ell'_2,\dots,\ell'_{\G})$ which implies $$\sum_{i=1}^{\G-1}\sum_{j=i+1}^{\G} \ell_i\ell_jd_{C_{\G}}(v_i,v_j)\geq \sum_{i=1}^{\G-1}\sum_{j=i+1}^{\G} \ell'_i\ell'_jd_{C_{\G}}(v_i,v_j).$$ 
\end{proof}

\section{Optimal unicyclic graphs with given degree sequence and fixed girth}\label{fixed girth}
 
 In this section, we will discuss the structure of the optimal unicyclic graphs that minimizes the Wiener index among unicyclic graphs with given degree sequence and fixed girth $\G$.
Note that the ``optimal'' graph may not be unique.

We first need the following lemma on centroids of unicyclic graphs:
\begin{lem} (\cite{miroslav})
	For a unicyclic graph $G$ with cycle $C_{\G}$, the centroid $M(G)$ is either a $K_1$ or $K_2$ or $M(G)\subset C_{\G}$. \label{centroid}
\end{lem}  

Let us build three special graphs, which possess different types of centroids.

\begin{defn}\label{greedyuni}
Let $D$ be a degree sequence of a unicyclic graph $G$, with fixed girth $\G$. The following construction leads to \textbf{$M(G) \subset C_{\G}$ and $M(G)$ is a $K_1$ or $K_2$}. We call such graph ``the greedy unicyclic graph", and it is denoted by $\Gr_u(D,\G)$. It is obtained from the subsequent ``greedy algorithm'':
\begin{itemize}
 \item[ i)] The vertex with largest degree will be on the cycle and will be labelled $v$ (the root);
 \item[ ii)] Label the two neighbors of $v$ on the cycle as $v_{1}$, $v_{2}$, and the neighbors of $v$ outside the cycle as $v_3, \dots,$ assign the largest degrees available to them such that $\deg(v_1) \geq \deg(v_2) \geq \deg(v_3)\dots$;
 \item[ iii)] Label the neighbors of $v_1$ and $v_2$ (except $v$) $v_{11},v_{12},v_{13},\dots$ 
 and $v_{21},v_{22}$ $,v_{23},\dots$, respectively, with $v_{11}$ and $v_{21}$ on the cycle. Assign the largest degrees available such that $$\deg(v_{11}) \geq \deg(v_{21})\geq \deg(v_{12})\geq \deg(v_{13})\geq \dots\deg(v_{21}) \geq\dots$$
 \item[ iv)] Then label the neighbors of $v_3, v_4, \dots$ such that they take all the largest degrees available and that  $\deg(v_{31}) \geq \deg(v_{32})\geq \dots\geq \deg(v_{41}) \geq \deg(v_{42})\geq \dots$. 
 
\noindent Note that the degree of a vertex on the cycle must be at least 2. Hence, we should keep some greater or equal to 2 degrees for the vertices on the cycle not yet labelled.
\item [v)]Repeat (iii) for all the newly labelled vertices on the cycle, always start with the neighbors of the labeled vertex with largest degree which are on the cycle and whose neighbors are not labeled yet.
\item [vi)] Repeat (iv) for the newly labelled vertices not on the cycle.
\item[vii)] When all the vertices on the cycle have been labelled, continue with the labelling as in Lemma \ref{prop2}.
\end{itemize}
\end{defn}
Figures \ref{fig:example 1} D, F and G are examples of such a graph.

\begin{defn}\label{CC}
Let $D$ be a degree sequence of a unicyclic graph $G$, with fixed girth $\G$. This configuration gives \textbf{$M(G)\subset C_{\G}$ and $M(G)\neq K_1,K_2$}.  We call such graph ``cycle-centered graph" and it is denoted by $\CC(D,\G)$ and is obtained as follows:
\begin{itemize}
 \item[ i)] We assign the largest degrees to all the vertices on the cycle $C_{\G}$ ordered according to Corollary \ref{Cor:deg}.
 \item[ ii)] We use Lemma \ref{prop2} to label the rest of the vertices not on the cycle.
 \end{itemize}

\end{defn}

Figures \ref{fig:example 1} C and E are examples of such a graph.
\vspace{2cm}

Let $D=(d_1,\dots,d_t,1,\dots,1)$, with $d_1\geq d_2 \geq \cdots \geq d_n$, and $d_t\geq 2$. We denote $D'$ the reduced degree sequence of $D$, i.e. $D'=(d_1,\dots,d_t)$. By the handshake lemma, the number of leaves is fixed. Therefore, the unicyclic graphs obtained from $D$ and $D'$ are the same. 


\begin{defn}\label{outgreedy}
Let $D$ be a degree sequence of a unicyclic graph $G$, with fixed girth $\G$. In this last configuration, \textbf{$M(G) \not \subset C_{\G}$}. The out-greedy unicyclic graph denoted by $\Gr_u^*(D,\G)$ is achieved through the following algorithm:
\begin{itemize}
\item [i)] Perform the greedy algorithm for a tree on the reduced degree sequence $D'$ and label the root as $v$. 
\item [ii)] Choose two vertices $u_1$ and $u_2$  which are either leaves or pseudo-leaves that belong to branches with the greatest number of vertices such that $d(u_1,u_2)=\G-1$. 
 Add an edge between them. If $u_1$ and $u_2$ are both leaves, the resulting graph will have $D$ as degree sequence. Otherwise, we remove a neighbor leaf from $u_1$ and/or $u_2$, respectively to keep the same degree sequence. Make sure that the choice of $u_1$ and $u_2$ will ``preserve'' the greedy algorithm.


\end{itemize}  
Note that for the centroid to be outside of the cycle we need, $\frac{\G}{2} < h(\Gr(D'))$, which implies $v \not \in P_{u_1u_2}$.
\end{defn}
Figures \ref{fig:example 1} A and B are examples of such a graph.

\begin{rem}
	Note that these definitions are not mutually exclusive.
	\begin{itemize}
		\item In fact if $\frac{\gamma}{2} << h(\Gr(D'))$, the definitions \ref{greedyuni} and \ref{CC} coincide with $M(G)\subset C_{\G}$.
		\item If $\frac{\gamma}{2} \geq h(\Gr(D'))$, the definitions \ref{greedyuni} and \ref{outgreedy} coincide with $M(G)=K_1,K_2$. 
	\end{itemize}
\end{rem}

\begin{theorem}\label{thm:fixedcycle}
Let $D$ be a degree sequence and $G \in \Ud^{\G}$, then
\[W(G)\geq \min\{W(\Gr_u(D,\G)),W(\CC(D,\G)),W(\Gr_u^*(D,\G))\}.\]
\end{theorem}

\begin{proof}
Let $G \in \Ud^{\G}$. We may assume the vertices on the cycle are ordered as in Theorem \ref{thm:cycle}. After removing the edge $e$ between the lowest degrees, we obtain a tree $T$. The following formula then holds:
\[W(G)=W(T_G)-\Di(G).\]
In order to minimize $W(G)$, we have to minimize $W(T_G)$, while still trying to maximize $\Di(G)$.

We need to discuss the three cases depending on the centroids of the graphs.

\textbf{Case 1: $M(G)\not \subset C_{\G} $}. In other words the centroid is either a vertex or an edge not on the cycle. Note that this case only happens if $\frac{\G}{2}<h(T_G)$. After removing the edge $e$, in order to minimize the Wiener index, all the paths have to satisfy the conditions in Lemma \ref{lem:diff} and the centroid of the tree will be the centroid of $G$. On the other hand, to maximize $\Di(G)$, we need to balance all the vertices on the cycle. The only possible option is to attach as much branches on the cycle and still preserving the ordering on each path. This is exactly the construction of the outside greedy unicyclic graph $\Gr_u^*(D,\G)$.

\textbf{Case 2: $M(G) \subset C_{\G}\, \text{and is a}\, K_1 \,\text{or}\,K_2  $}
We still have the ordering of all the other paths not from the cycle following Lemma \ref{lem:diff} and even for the path obtained from the cycle. This is exactly the construction of the greedy unicyclic graph $\Gr_u(D,\G)$.

\textbf{Case 3: $M(G) \subset C_{\G}\, \text{different from}\, K_1 \,\text{or}\,K_2  $}
From Remark \ref{rem:diff} and Lemma \ref{prop2}, the unicyclic graph that maximizes $\Di(G)$, while having the least $W(T)$ is $\CC(D,\G)$.
\end{proof}

Figures \ref{fig:example 1} and \ref{fig:example 2} show the different optimal graphs for a given degree sequence with fixed girth. We can see that depending on the sequence and the girth, the minimum Wiener index is achieved by one of our three special unicyclic graphs.

\begin{figure}[H]
	\centering
	\begin{subfigure}[b]{0.4\linewidth}
		\includegraphics[width=\linewidth]{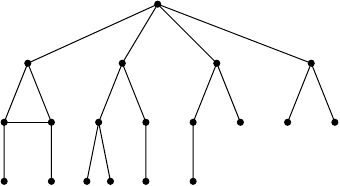}
		\caption{$W(\Gr_u^*(D_1,3))=598$}
	\end{subfigure}
	\begin{subfigure}[b]{0.4\linewidth}
		\includegraphics[width=\linewidth]{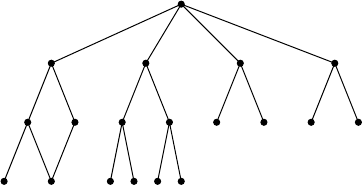}
		\caption{$W(\Gr_u^*(D_1,4))=590$}
	\end{subfigure}
	\begin{subfigure}[b]{0.4\linewidth}
		\includegraphics[width=\linewidth]{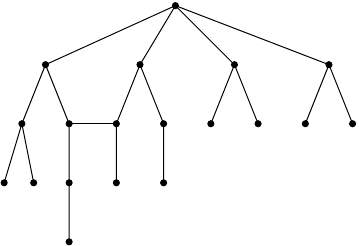}
		\caption{$W(\CC(D_1,5))=575$}
	\end{subfigure}
	\begin{subfigure}[b]{0.4\linewidth}
		\includegraphics[width=\linewidth]{C6_4333333322_572gr.pdf}
		\caption{$W(\Gr_u(D_1,6))=572$}
	\end{subfigure}
	\begin{subfigure}[b]{0.4\linewidth}
		\includegraphics[width=\linewidth]{C6_4333333322_572cy.pdf}
		\caption{$W(\CC(D_1,6))=572$}
	\end{subfigure}
	\begin{subfigure}[b]{0.4\linewidth}
		\includegraphics[width=\linewidth]{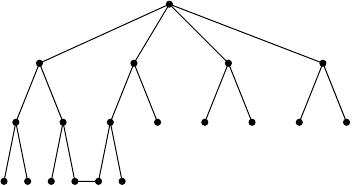}
		\caption{$W(\Gr_u(D_1,7))=565$}
	\end{subfigure}
	\begin{subfigure}[b]{0.4\linewidth}
		\includegraphics[width=\linewidth]{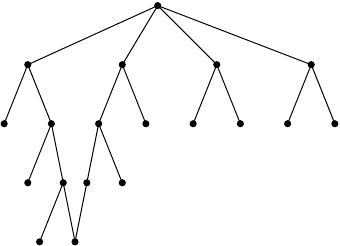}
		\caption{$W(\Gr_u(D_1,8))=576$}
	\end{subfigure}
	\caption{Optimal graphs for the sequence $D_1=(4,3,3,3,3,3,3,3,2,2,1,\dots,1)$ with fixed $\G$ and their respective Wiener indices.}
	\label{fig:example 1}
\end{figure}

\begin{figure}
	\centering
	\begin{subfigure}[b]{0.4\linewidth}
		\includegraphics[width=\linewidth]{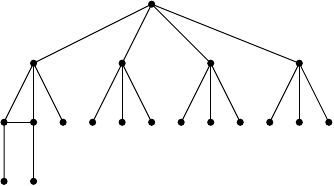}
		\caption{$W(\Gr_u^*(D_2,3))=538$}
	\end{subfigure}
	\begin{subfigure}[b]{0.4\linewidth}
		\includegraphics[width=\linewidth]{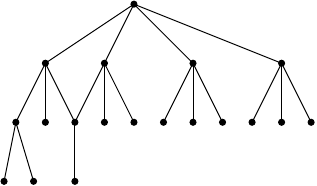}
		\caption{$W(\Gr_u(D_2,4))=538$}
	\end{subfigure}
	\hspace{1cm}
	\begin{subfigure}[b]{0.23\linewidth}
		\includegraphics[width=\linewidth]{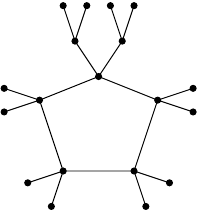}
		\caption{$W(\CC(D_2,5))=519$}
	\end{subfigure}
	\hspace{1cm}
	\begin{subfigure}[b]{0.23\linewidth}
		\includegraphics[width=\linewidth]{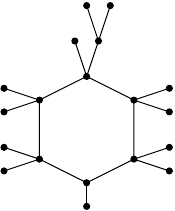}
		\caption{$W(\CC(D_2,6))=528$}
	\end{subfigure}
	\hspace{1cm}
	\begin{subfigure}[b]{0.23\linewidth}
		\includegraphics[width=\linewidth]{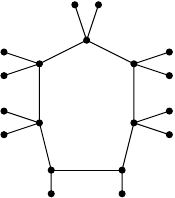}
		\caption{$W(\CC(D_2,7))=523$}
	\end{subfigure}
	
	\caption{Optimal graphs for the sequence $D_2=(4,4,4,4,4,4,3,3,1,\dots,1)$ with fixed $\G$ and their respective Wiener indices.}
	\label{fig:example 2}
\end{figure}

Now, what will happen if we don't fix the girth? For a degree sequence of the form $(d_1,2,2,\dots,2,1,\dots,1)$ with $d_1\geq 3$, the unicyclic graph that minimizes the Wiener index has the largest possible girth. For other degree sequences, even though larger girth initially implies a smaller Wiener index, there will be a point where it is no longer best to increase the girth, as seen in Figure \ref{fig:example 1} from F to G, or in Figure \ref{fig:example 2} from C to D. For sure, we can easily rule out the out-greedy unicyclic graph $\Gr_u^*(D,\G)$  as a canditate for our optimal unicyclic graph. In fact, it is easy to see that one of the other two candidates will have a smaller Wiener index if the girth is not fixed. However the best girth will depend on the interplay of the loss and gain mentioned in Section 5. We have the following conjecture:

\begin{conjecture}\label{thm:main}
Let $D=(d_1,d_2,\dots,d_{t-1},d_{t},1,\dots,1)$ such that $d_{t}\geq 2$ and $d_2\neq 2$ be a degree sequence and $G$ a unicyclic graph with degree sequence $D$. Then \[W(G)\geq \min\{\Gr_u(D,\G^*),\CC(D,\G^*)\}\]
where $\G^*=\begin{cases}2*h(\Gr(D'))-1 & \text{if  $h(\Gr(D'))$ is achieved by a unique leaf}\\2*h(\Gr(D'))& \text{otherwise}\end{cases}$,

 where $\Gr(D')$ is the greedy tree obtained from the degree sequence $D'=(d_1,d_2,\dots,d_{t-1}-1,d_{t}-1,1,\dots,1)$.
\end{conjecture}

%
%
%

\bibliographystyle{abbrv} 
\bibliography{wieneruni}
\end{document}